\documentclass[11pt,english]{article}

\usepackage[margin = 2.5cm]{geometry}

\usepackage{amsthm}
\usepackage{amsmath}
\usepackage{amssymb}
\usepackage{mathtools}
\usepackage{graphicx}
\usepackage[hidelinks]{hyperref}
\usepackage{enumitem}

\usepackage{caption}% http://ctan.org/pkg/caption
\usepackage[square,sort,comma,numbers]{natbib} 
\usepackage{setspace}
\usepackage{cleveref}
\theoremstyle{plain}

\newtheorem*{thm*}{Theorem}
\newtheorem{thm}{Theorem}[section]
\Crefname{thm}{Theorem}{Theorems}

\newtheorem*{lem*}{Lemma}
\newtheorem{lem}[thm]{Lemma}
\Crefname{lem}{Lemma}{Lemmas}

\newtheorem*{claim*}{Claim}
\newtheorem{claim}[thm]{Claim}
\crefname{claim}{Claim}{Claims}
\Crefname{claim}{Claim}{Claims}

\Crefname{prop}{Proposition}{Propositions}

\newtheorem{cor}[thm]{Corollary}
\crefname{cor}{Corollary}{Corollaries}

\crefname{conj}{Conjecture}{Conjectures}

\Crefname{qn}{Question}{Questions}

\Crefname{obs}{Observation}{Observations}

\Crefname{ex}{Example}{Examples}

\theoremstyle{definition}

\Crefname{prob}{Problem}{Problems}

\Crefname{defn}{Definition}{Definitions}

\theoremstyle{remark}

\usepackage{xpatch}
\xpatchcmd{\proof}{\itshape}{\normalfont\proofnamefont}{}{}
\newcommand{\proofnamefont}{}
\renewcommand{\proofnamefont}{\bfseries}

\usepackage{framed}

\usepackage{soul}
\usepackage{color}

\newcommand{\remove}[1]{}

\newcommand{\ceil}[1]{
    \lceil #1 \rceil
}

\newcommand{\Z}{\mathbb{Z}}
\newcommand{\R}{\mathbb{R}}

\newcommand{\G}{\mathcal{G}}

\usepackage{bbm}

\title{Fractional triangle decompositions in almost complete graphs}

\author{
	Vytautas Gruslys\thanks{
		Email: \texttt{vytautas.gruslys}@\texttt{gmail.com}.
	}
	\and
    Shoham Letzter\thanks{
		Department of Mathematics, 
		University College London, 
		Gower Street, London WC1E~6BT, UK. 
		Email: \texttt{s.letzter}@\texttt{ucl.ac.uk}. 
		Research supported by the Royal Society.
    }
}

\newcommand{\codeg}{\bar{d}}
\newcommand{\boldd}{\mathbf{d}}

\begin{document}

\date{}
\maketitle

\begin{abstract}

	We prove that every $n$-vertex graph with at least $\binom{n}{2} - (n - 4)$ edges has a fractional triangle decomposition, for $n \ge 7$. This is a key ingredient in our proof, given in a companion paper, that every $n$-vertex $2$-coloured complete graph contains $n^2/12 + o(n^2)$ edge-disjoint monochromatic triangles, which confirms a conjecture of Erd\H{o}s.

	\setlength{\parskip}{\medskipamount}
    \setlength{\parindent}{0pt}
    \noindent

\end{abstract}

	\section{Introduction}

		A \emph{triangle packing} in a graph $G$ is a collection of edge-disjoint triangles, and a \emph{triangle decomposition} is a triangle packing that covers all the edges. A \emph{fractional triangle packing} in a graph $G$ is an assignment of weights in $[0,1]$ to the triangles in $G$, such that the total weight of every edge is at most $1$; namely, $\sum_{w \in V(G)} \omega(uvw) \le 1$ for every edge $uv$ (where $\omega(uvw) = 0$ if $uvw$ is not a triangle).
		Given a triangle packing $\omega$ and an edge $e = uv$ in $G$, we define $\omega(e) = \sum_{w \in V(G)} \omega(uvw)$; so $\omega(e) \le 1$. A \emph{fractional triangle decomposition} in a graph $G$ is a fractional triangle packing $\omega$, satisfying that $\omega(e) = 1$ for every edge $e$.
		Our main result in this paper is the following theorem, which shows that almost complete graphs have fractional triangle decompositions.

		\begin{thm} \label{thm:n-four}
			Let $G$ be a graph on $n \ge 7$ vertices with $e(G) \ge \binom{n}{2} - (n - 4)$. Then there is a fractional triangle decomposition in $G$.
		\end{thm}

		\Cref{thm:n-four} is tight in two ways: the complete graph on six vertices with two edges removed (intersecting or not) does not have a fractional triangle decomposition; and the graph on vertex set $[n]$ with non-edges $\{ xn : x \in \{4, \ldots, n\!-\!1\}\} \cup \{12\}$ is an $n$-vertex graph with $n-3$ non-edges that does not have a fractional triangle decomposition.

		Our main motivation for proving \Cref{thm:n-four} is our \cite{us1} proof that every $n$-vertex $2$-coloured complete graph has $n^2/12 + o(n^2)$ edge-disjoint monochromatic triangles, which confirms a conjecture of Erd\H{o}s \cite{erdos}. To prove the conjecture, we use a reduction to fractional monochromatic triangle packings, due to Haxell and R\"odl \cite{haxell-rodl}. Our proof there is inductive, and \Cref{thm:n-four} is a key ingredient in the induction step.

		A well-known conjecture of Nash-Williams \cite{nash-williams} asserts that every $n$-vertex graph $G$ with minimum degree at least $3n/4$, where $n$ is large and $G$ satisfies certain `divisibility conditions', has a triangle decomposition. While this conjecture is still open, significant progress towards it has been made. Recently, Delcourt and Postle \cite{delcourt-postle} showed that every $n$-vertex graph with minimum degree at least $0.83 n$ has a fractional triangle decomposition, improving on several previous results (see, e.g., \cite{gustavsson,yuster05,dukes,garaschuk,dross}).
		Combined with a result of Barber, K\"uhn, Lo and Osthus \cite{barber-et-al}, it follows that the statement obtained by replacing $3/4$ by $0.831 n$ in Nash-Williams's conjecture is true. 
		Delcourt and Postle's result (or any result about fractional triangle decompositions in graphs with large minimum degree) can be used to prove \Cref{thm:n-four} for sufficiently large $n$. 
		However, crucially, in \cite{us1} we need \Cref{thm:n-four} to hold for all $n \ge 7$, and we thus prove \Cref{thm:n-four} without relying on such results. 

		In fact, in \cite{us1} we use the following stronger version of \Cref{thm:n-four}.

		\begin{cor} \label{cor:n-four} 
			Let $G$ be a complete graph on $n \ge 7$ vertices, and let $\phi : E(G) \to [0,1]$ be such that $\sum_{e \in E(G)} \phi(e) \ge \binom{n}{2} - (n-4)$. 
			Then there is a fractional triangle packing $\omega$ in $G$ such that $\omega(e) = \phi(e)$ for every $e \in E(G)$.
		\end{cor}

		In order to prove \Cref{thm:n-four}, we prove a stronger statement (see \Cref{thm:n-four-a}) by induction, constructing a suitable fractional triangle packing in $G$ using fractional triangle packings of certain graphs related to $G$ on $n-1$ and $n-2$ vertices. The induction base is proved by computer search.
		\Cref{cor:n-four} follows from \Cref{thm:n-four} via a reduction from weighted graphs to simple graphs (see \Cref{lem:integer-to-fractional}).

		\subsection*{Organisation of the paper}

			In \Cref{sec:prelims} we introduces some notation, mention a few preliminaries, and state \Cref{thm:n-four-a} -- a strengthening of \Cref{thm:n-four} which is more amenable to an inductive proof. In \Cref{sec:frac} we prove \Cref{lem:integer-to-fractional}, which will allow us to prove \Cref{cor:n-four} and will be handy for the proof of \Cref{thm:n-four-a}. In \Cref{sec:computer} we describe the algorithm used in our computer search, and explain how it proves \Cref{thm:n-four,thm:n-four-a} for small values of $n$. Finally, in \Cref{sec:proof} we complete the proof of \Cref{thm:n-four-a}.

	\section{Preliminaries} \label{sec:prelims}

		Recall that a \emph{fractional triangle packing} in $G$ is an assignment $\omega$ of weights in $[0, 1]$ to the triangles in $G$, such that the total weight on every edge of $G$ is at most $1$; i.e.\ $\sum_{w \in V(G)} \omega(uvw) \le 1$ for every edge $uv$ in $G$ (where $w(uvw) = 0$ whenever $uvw$ is not a triangle). For every edge $uv$ we define $\omega(uv) := \sum_{w \in V(G)} \omega(uvw)$. A \emph{fractional triangle decomposition} in a graph $G$ is a fractional triangle packing $\omega$ satisfying $\omega(e) = 1$ for every edge $e$ in $G$.
		
		The \emph{uncovered weight} in $\omega$ is the total uncovered edge-weight, namely $\sum_{e \in E(G)} (1 - \omega(e))$. (So a fractional triangle packing $\omega$ is a fractional triangle decomposition if and only if the uncovered weight in $\omega$ is $0$.)
		Given a graph $G$, the number of \emph{missing edges} in $G$ is the number of pairs of vertices that are not edges of $G$. 		

		In order to prove our main result, \Cref{thm:n-four}, we prove the following stronger result. We note that the lower bound on $n$ is tight, as there exists a graph on $10$ vertices with $10$ missing edges, that does not have a fractional triangle packing with uncovered weight at most $4$.

		\begin{thm} \label{thm:n-four-a}
			Let $G$ be a graph on $n$ vertices with at most $n - 4 + a$ missing edges, where $n \ge 11$ and $0 \le a \le 4$. 
			Then there is a fractional triangle packing in $G$ with total uncovered weight at most $a$, such that every triangle has weight at most $1/2$.
		\end{thm}	

		We prove \Cref{thm:n-four,thm:n-four-a} for $n \le 13$ by computer. More precisely, we prove the following lemma. For a description of our algorithm, and a proof of this lemma using the outcome of the computer search, see \Cref{sec:computer}. The certificates relevant to the computer search can be found \href{https://liveuclac-my.sharepoint.com/:f:/g/personal/ucahsle_ucl_ac_uk/EmF10ks9B91ApHhBWyKIu0oBzCRu6j3m-lcfvMgrqLq5hA?e=5\%3aJtRVTt&at=9}{\textcolor{blue}{here}}.

		\begin{lem} \label{lem:computer}
			Let $G$ be an $n$-vertex graph with $n - 4 + a$ missing edges, where $a \in \{0, \ldots, 4\}$. 
			\begin{itemize} 
				\item
					If $n \in \{11, 12, 13\}$ and $a \in \{0, \ldots, 4\}$, then $G$ has a fractional triangle packing with uncovered weight at most $a$, such that every triangle in $G$ has weight at most $1/2$.
				\item
					If $n \in \{7, \ldots, 10\}$ and $a = 0$, then $G$ has a fractional triangle decomposition.
			\end{itemize}
		\end{lem}

		It will be convenient for us to assume that $G$ has \emph{exactly} $n-4+a$ missing edges, for some $a \in \{0, \ldots, 4\}$. To do so, we use the following reduction. Note that \Cref{thm:n-four} follows directly from \Cref{thm:n-four-a} and \Cref{lem:computer,lem:reduction-n-4-non-edges}.

		\begin{lem} \label{lem:reduction-n-4-non-edges}
			Suppose that every graph on $n$ vertices with exactly $m < \binom{n}{2}$ missing edges has a fractional triangle decomposition, such that every triangle has weight at most $\beta \ge 1/3$. Then the same holds for graphs with \emph{at most} $m$ missing edges.
		\end{lem}

		\begin{proof} 
			We prove by induction that every $n$-vertex graph with $k$ missing edges, where $k \le m$, has a fractional triangle decomposition, such that every triangle has weight at most $\beta$. The case $k = m$ holds by assumption. Now suppose that the statement holds for $k$ with $1 \le k \le m$. Let $G$ be an $n$-vertex graph with $k-1$ missing edges. Note that $G$ has a triangle. Indeed, by assumption on $k$, the graph obtained by removing any edge from $G$ has a fractional triangle decomposition; in particular, it contains a triangle (as $G$ has at least one edge). Let $uvw$ be a triangle in $G$. By assumption on $k$,  each of the graphs $G \setminus \{uv\}$, $G \setminus \{uw\}$ and $G \setminus \{vw\}$ has a fractional triangle decomposition, such that the weight of each triangle is at most $\beta$. Taking the average of these three packings, and additionally assigning weight $1/3$ to $uvw$, we obtain a triangle decomposition in $G$ with no heavy triangles, as required.
		\end{proof}

		A \emph{weighted graph} is a pair $(G, \phi)$ where $G$ is a graph and $\phi$ is an assignment of weights in $[0, 1]$ to the edges of $G$. The \emph{missing weight} in $(G, \phi)$ is $\sum_{e \in V(G)^{(2)}} (1 - \phi(e))$, where $\phi(e) = 0$ if $e$ is not an edge of $G$. A fractional triangle packing $\omega$ in $(G, \phi)$ is a fractional triangle packing $\omega$ in $G$ such that $\omega(e) \le \phi(e)$ for every edge $e$ in $G$. The \emph{uncovered weight} in $\omega$ is, as above, the total uncovered edge-weight, namely $\sum_{e \in E(G)}(\phi(e) - \omega(e))$.
		
		Our proof of \Cref{thm:n-four-a} is inductive. When applying the induction step, it will be useful for us to have a version of \Cref{thm:n-four-a} for weighted graphs. This can be achieved by the following lemma.

		\begin{lem} \label{lem:integer-to-fractional}
			Suppose that every graph on $n$ vertices with at most $m$ missing edges has a fractional triangle packing with total uncovered weight at most $a$, such that every triangle has weight at most $\beta$. Then every \emph{weighted} graph with missing weight at most $a$, has a fractional triangle packing $\omega$ with uncovered weight at most $a$ such that every triangle has weight at most $\beta$.
		\end{lem}

		Note that \Cref{cor:n-four} follows immediately from \Cref{thm:n-four} and \Cref{lem:integer-to-fractional}.

		In our proof of \Cref{thm:n-four-a} we make use of the following corollary of Ore's theorem \cite{ore}, which asserts that if a graph $G$ on $n \ge 3$ vertices satisfies $d(u) + d(w) \ge n$ for every two non-adjacent vertices $u$ and $w$, then $G$ has a Hamilton cycle.

		\begin{cor} \label{cor:hamilton}
			Let $G$ be a graph on $n \ge 3$ vertices with at most $n - 3$ missing edges. Then $G$ has a Hamilton cycle.
		\end{cor}

		A \emph{heavy triangle} in a fractional triangle packing $\omega$ is a triangle $T$ with $\omega(T) > 1/2$. Given a graph $G$ on $n$ vertices and a vertex $u$ in $G$, we denote the degree of $u$ by $d_G(u)$ and its non-degree (namely, the number of non-edges incident with $u$) by $\codeg_G(u)$; so $\codeg_G(u) = n - 1 - d_G(u)$. When $G$ is clear from the context, we omit the subscript $G$.

	\section{Fractional triangle packings in weighted graphs} \label{sec:frac}

		In this section we prove \Cref{lem:integer-to-fractional}, which reduces the problem of finding large fractional packings in weighted graphs, to finding such packings in unweighted graphs.

		\begin{proof}[Proof of \Cref{lem:integer-to-fractional}]
			Let $(G, \phi)$ be a weighted graph as in the statement of the claim.
			Suppose first that $\phi(e)$ is rational for every edge $e$, and let $r$ be such that $\phi(e) r$ is integer for every edge $e$. We make use of the following claim.

			\begin{claim} \label{claim:partition}
				Suppose that $d_1, \ldots, d_N \in \{0, \ldots, r\}$ satisfy $d_1 + \ldots + d_N \le r \cdot m$.
				Then there exist subsets $S_1, \ldots, S_r \subseteq [N]$ of size at most $m$ such that every $i \in [N]$ appears in exactly $d_i$ sets $S_j$ with $j \in [r]$.
			\end{claim}

			\begin{proof} 
				We prove the claim by induction on $r$.
				If $r = 0$, the statement holds trivially.
				Suppose that $r \ge 1$, and that the statement holds for $r - 1$.
				Without loss of generality, suppose that $d_1 \ge \ldots \ge d_N$. 
				Let $S_r$ be the set of indices $i \in [m]$ with $d_i \ge 1$.
				Define
				\begin{equation*}
					d_i' = \left\{
						\begin{array}{ll}
							d_i - 1 & i \in S_r \\
							d_i & \text{otherwise.}
						\end{array}
					\right.
				\end{equation*}
				Note that $d_i' \le r-1$ for every $i \in [N]$. Indeed, otherwise, $d_1, \ldots, d_{m+1} \ge r$, contradicting the assumption that $\sum_{i \in [N]} d_i \le r \cdot m$. Moreover, $\sum_{i \in [N]} d_i' \le (r-1)m$. Indeed, if $|S_r| = m$ then $\sum_{i \in [N]} d_i' = \sum_{i \in [N]} d_i - m \le (r-1)m$; and if $|S_r| < m$ then $d_m' = \ldots = d_N' = 0$, so $\sum_{i \in [N]} d_i' = \sum_{i \in [m-1]}d_i' < (r-1)m$.
				It follows that, by induction on $r$, there exist sets $S_1, \ldots, S_{r-1} \subseteq [N]$ of size at most $m$, such that every $i \in [N]$ is in exactly $d_i'$ sets $S_j$ with $j \in [r-1]$. The sets $S_1, \ldots, S_r$ satisfy the requirements for $d_1, \ldots, d_N$.
			\end{proof}

			Note that $(1 - \phi(e)) r \in \{0, \ldots, r\}$ for every edge $e$, and $\sum_{e \in E(G)} (1 - \phi(e))r \le r \cdot m$. Thus, by \Cref{claim:partition}, there exist sets $S_1, \ldots, S_r \subseteq E(G)$ of size at most $m$, such that every $e \in E(G)$ is in exactly $(1 - \phi(e))r$ sets $S_i$. Let $G_i$ be the graph on vertex set $V(G)$ with non-edges $S_i$. Then $G_i$ is a graph on $n$ vertices with at most $m$ non-edges, so by assumption there is a fractional triangle packing $\omega_i$ in $G_i$ with uncovered weight at most $a$, such that all triangles have weight at most $\beta$. Let $\omega = (1/r) \cdot \sum_{i \in [r]} \omega_i$. Then $\omega$ is a triangle packing in $(G, \phi)$ (as $\omega(e) \le (r - (1 - \phi(e))r)/r = \phi(e)$) with uncovered weight at most $a$ such that all triangles have weight at most $\beta$. This concludes the proof in the case where $\phi(e)$ is rational for every $e \in E(G)$.

			Now consider the general case, where $\phi(e)$ may be irrational for some edges $e$. For $\ell \in \mathbb{N}$, let $\phi_{\ell}$ be such that $\phi_{\ell}(e)$ is rational for every edge $e$, and $\phi(e) \le \phi_{\ell}(e) \le \min\{1, \phi(e) + 1/{\ell}\}$. Then by the proof for rational edge-weightings, there is a fractional triangle packing $\omega_{\ell}(e)$ with the requirements stated in the claim. By taking the limit of taking the limit of a converging subsequence of $(\omega_{\ell})_{\ell}$, we find a fractional triangle packing $\omega$ that satisfies the requirements for $\phi$.
		\end{proof}

	\section{Computer search} \label{sec:computer}

		In this section we describe the algorithms that we use to prove \Cref{lem:computer}.
		The certificates relevant to the computer search can be found \href{https://liveuclac-my.sharepoint.com/:f:/g/personal/ucahsle_ucl_ac_uk/EmF10ks9B91ApHhBWyKIu0oBzCRu6j3m-lcfvMgrqLq5hA?e=5\%3aJtRVTt&at=9}{\textcolor{blue}{here}}.

		We say that a pair $(N, M)$ of integers which is \emph{relevant} if
		\begin{itemize} 
			\item
				either $N \in \{11, 12, 13\}$ and $M = \binom{N}{2} - (N - 4 + a)$ for some $a \in \{0, \ldots, 4\}$, 
			\item
				or $N \in \{7, \ldots, 10\}$ and $M = \binom{N}{2} - (N - 4)$.
		\end{itemize}

		\subsection{The algorithm}
			The algorithm receives a pair of integers $(N, M)$ as input. 
			It then performs the following steps.
			\begin{enumerate} 
				\item \label{itm:step1}
					Generate all sequences of integers $(d_1, \ldots, d_N)$ such that 
					\begin{itemize} 
						\item
							$d_1 \ge \ldots \ge d_N$,
						\item
							$\sum_{i \in [N]} d_i = 2M$,
						\item
							there exists a graph on $N$ vertices with degree sequence $(d_1, \ldots, d_N)$.
					\end{itemize}
				\item \label{itm:step2}
					Form an auxiliary acyclic digraph $D$ as follows. 
					\begin{itemize} 
						\item
							The vertices of $D$ are degree sequences of graphs $\boldd = (d_1, \ldots, d_n)$ such that $d_1 \ge \ldots \ge d_n$.
						\item
							The \emph{sinks} (namely the vertices of out-degree $0$) are the degree sequences generated in step \ref{itm:step1}.
						\item
							For every $\boldd = (d_1, \ldots, d_n)$ in $F$, the collection of in-neighbours is defined as follows.
							\begin{enumerate}[label = \rm(\alph*), ref = 2\rm(\alph*)] 
								\item \label{itm:step2a}
									If $\boldd$ is the empty sequence, it has no in-neighbours.
								\item \label{itm:step2b}
									If $\sum_{i \in [n]} > \frac{1}{2} \binom{n}{2}$, set $\boldd' = (n-1-d_n, \ldots, n-1-d_1)$, and let $\boldd'$ be the only in-neighbour of $\boldd$.
								\item \label{itm:step2c}
									Otherwise, if $d_n \in \{0, 1\}$, let $i$ be the least integer satisfying $d_{i+1} \le 1$.
									Let the in-neighbourhood of $\boldd$ consist of sequences $\boldd' = (d_1', \ldots, d_i')$ such that
									\begin{itemize} 
										\item
											$d_1' \ge \ldots \ge d_i'$,
										\item
											$d_j' \le d_j$ for every $j \in [i]$,
										\item
											$\sum_{j \in [i]} d_j - \sum_{j \in [i]} d_j' \le \sum_{j \in \{i+1, \ldots, n\}} d_j$,
										\item
											$\boldd'$ is a degree sequence of a graph.
									\end{itemize}
								\item \label{itm:step2d}
									If neither of the previous conditions hold, let the in-neighbours of $\boldd$ be the sequences $\boldd' = (d_1', \ldots, d_{n-1}')$ satisfying
									\begin{itemize} 
										\item
											$d_1' \ge \ldots \ge d_{n-1}'$,
										\item
											$d_j' \in \{d_j-1, d_j\}$ for every $j \in [n-1]$,
										\item
											$\sum_{j \in [n-1]} d_j - \sum_{j \in [n-1]} d_j' = d_1$,
										\item
											$\boldd'$ is a degree sequence of a graph.
									\end{itemize}
							\end{enumerate}
						\item
							In particular, the only \emph{source} (namely vertex of in-degree $0$) is the empty sequence. 
					\end{itemize}
				\item \label{itm:step3}
					For each $\boldd \in D$, we generate the collection $\G(\boldd)$ of all graphs with degree sequence $\boldd$, as follows.
					\begin{itemize} 
						\item
							For $\boldd$ being the empty set, we set $\G(\boldd)$ to consist of the empty graph with empty vertex set.
						\item
							Suppose that we have calculated $\G(\boldd)$ for some $\boldd = (d_1, \ldots, d_n) \in V(D)$. Then for every out-neighbour $\boldd'$ of $\boldd$, let $\G(\boldd')$ consist of all graphs $G'$ with degree sequence $\boldd'$, that can be obtained from some $G \in \G(\boldd)$ by
							\begin{itemize}
								\item
									taking the complement of $G$, if the edge $\boldd \boldd'$ was formed in step \ref{itm:step2b},
								\item
									adding some new vertices to $G$ and joining each of them with at most one (new or existing) vertex, if $\boldd \boldd'$ was formed according to step \ref{itm:step2c},
								\item
									adding a new vertex to $G$ and joining it to at least two existing vertices, if $\boldd \boldd'$ was formed according to step \ref{itm:step2d}.
							\end{itemize}
					\end{itemize}
				\item \label{itm:step4}
					For each sink $\boldd$ in $D$ (so $\boldd$ is a degree sequence of an $N$-vertex graph with $M$ edges), and for each graph $G \in \G(\boldd)$, run a linear program to minimise the uncovered weight of a fractional triangle packing in $G$ with no heavy triangles (i.e.\ with no triangles of weights larger than $1/2$).
			\end{enumerate}

			{\bf Outcome.}\,
			For every relevant $(N, M)$, for sinks in the graph $D$ generated for $(N, M)$, all the graphs in $\G(\boldd)$ have a fractional triangle packing with uncovered weight at most $a$ (where $M = \binom{N}{2} - (N - 4 + a)$) with no heavy triangles.

		\subsection{Proof of \Cref{lem:computer}}

			It is now easy to prove \Cref{lem:computer}.
			\begin{proof} %[Proof of \Cref{lem:computer}]
				Fix some $(N, M)$, and let $D$ be the directed graph generated by the algorithm for $(N, M)$. It is easy to see that $\G(\boldd)$ is the collection of all graphs with degree sequence $\boldd$, for every vertex $\boldd$ in $D$. Indeed, this can be done by induction, noting that every graph $G$ with degree sequence $\boldd$ can be obtained from a graph $G'$ with degree sequence $\boldd'$ for some in-neighbour $\boldd'$ of $\boldd$, as described in step \ref{itm:step3}.
				In particular, the union of the families $\G(\boldd)$ over all sinks $\boldd$ of $D$ is the collection of all graphs on $N$ vertices with $M$ edges, using the fact that the set of sinks is the set of degree sequences of such graphs, by steps \ref{itm:step1} and \ref{itm:step2}.
				It thus follows from the outcome of the algorithm that every $N$-vertex graph with $M$ edges has a fractional triangle packing with uncovered weight at most $a$ and no heavy triangles, where $M = \binom{N}{2} - (N-4+a)$, and $(N, M)$ is relevant.
				\Cref{lem:computer} follows.
			\end{proof}

		\subsection{Remarks}

			We conclude this section with some remarks regarding the algorithm.

			\begin{enumerate} 
				\item
					In order to determine if a sequence $\boldd = (d_1, \ldots, d_n)$, where $d_1 \ge \ldots \ge d_n$, is a degree sequence of a graph, we use the well-known criterion due to Erd\H{o}s and Gallai \cite{erdos-gallai}, according to which $\boldd$ is a degree sequence of a graph if and only if $\sum_{i \in [n]} d_i$ is even, and 
					\begin{equation*}
						\sum_{i \in [k]}d_i \le k(k-1) + \sum_{i \in \{k+1, \ldots, n\}} \min\{k, d_i\}
					\end{equation*}
					for every $k \in [n]$.
				\item
					For correctness, it is not necessary to allow for edges of $D$ as in steps \ref{itm:step2b} and \ref{itm:step2c}. We do include such edges in $D$, as this means that we will mostly consider degree sequences of relatively sparse graphs (after `taking the complements' of degree sequences corresponding to graphs on $N$ vertices with $M$ edges). Such graphs are likely to have many leaves and isolated vertices, and removing them all at once, rather than one by one, decreases the size of $D$ and thus lets the algorithm to run faster.
				\item
					When forming the collections $\G(\boldd)$ in step \ref{itm:step3}, we adapt an algorithm of McKay and Piperno \cite{mckay-piperno} to detect whether a newly generated graph is isomorphic to a graph that was generated previously. 
				\item
					In principle, the fractional triangle packings found in step \ref{itm:step4} may be susceptible to rounding errors. To account for this possibility we find a rational approximation of the packings found, using continuous fractions approximations (while ensuring that the weights are non-negative, and that no edge receives weight larger than $1$). In practice, the program did not encounter any issues related to rounding errors. Nevertheless, for correctness, this had to be checked.
			\end{enumerate}

	\section{The proof} \label{sec:proof}

		In this section we prove \Cref{thm:n-four-a}.
		
		\begin{proof}%[Proof of \Cref{thm:n-four-a}] 
			We prove the result by induction on $n$. 
			By \Cref{lem:reduction-n-4-non-edges}, it suffices to prove the statement for graphs with exactly $n-4+a$ non-edges, where $a \in \{0, \ldots, 4\}$. 
			The case $n \in \{11, 12, 13\}$ thus follows from \Cref{lem:computer}.
			Let $n \ge  14$ and suppose that the statement of \Cref{thm:n-four-a} holds for $n - 1$ and $n - 2$.
			Let $G$ be a graph on $n \ge 14$ vertices with exactly $n - 4 + a$ non-edges, where $a \in \{0, \ldots, 4\}$.
			We consider four cases: there is a vertex $u$ with $\codeg(u) > (n+a)/3$; $m \in \{0, 1, 2, 3\}$; $a < 4$ and $m \ge 4$; and $a = 4$ and $m \ge 4$. The latter two carry the main difficulty of the proof.
			
			\subsection{Case 1. There is a vertex $u$ with $\codeg(u) > (n+a)/3$}

				Write $\codeg := \codeg(u)$. Consider the graph $G[N(u)]$; it has $d: = |N(u)| = n - 1 - \codeg$ vertices and at most $n - 4 + a - \codeg = d - 3 + a$ missing edges. By \Cref{cor:hamilton}, if $d \ge 3$, there is Hamilton cycle in $G[N(u)]$ with $\alpha \le a$ missing edges; i.e.\ there is an ordering $x_1, \ldots, x_d$ of the vertices in $N(u)$, such that $x_i x_{i+1}$ is an edge in $G$ for all but $\alpha$ values of $i \in [d]$ (addition is taken modulo $d$). Let $\omega'$ be the fractional triangle packing that gives each triangle $ux_ix_{i+1}$ with $x_i x_{i+1} \in E(G)$ weight $1/2$, and let $G'$ be the weighted graph obtained from $G \setminus \{u\}$ by giving $x_i x_{i+1}$ weight $1/2$ whenever $x_i x_{i+1}$ is an edge of $G$; giving weight $1$ to every other edge of $G \setminus \{u\}$; and giving non-edges weight $0$. 
				The total missing weight in $G'$ is at most
				\begin{align*} 
					n - 4 + a - \codeg + (d-\alpha)/2 
					& = n - 4 + a - \alpha/2 + (n - 1 - \codeg)/2 - \codeg \\
					& = 3n/2 -4.5 + a - \alpha/2 - 3\codeg/2 \\
					& \le (3n/2 - 4.5 + a - \alpha/2) - (n + \alpha + 1)/2 \\
					& = n - 5 + a - \alpha \\
					& = (n-1) - 4 + (a - \alpha),
				\end{align*}
				using $3\codeg \ge n + a + 1 \ge n + \alpha + 1$ for the inequality.
				By the induction hypothesis together with \Cref{lem:integer-to-fractional}, it follows that there is a fractional triangle packing in $G'$ with no heavy triangles and with uncovered weight at most $a - \alpha$. Combining this packing with $\omega$, we obtain a fractional triangle packing of $G$ with uncovered weight at most $a$ and no heavy triangles, as required.

				It remains to consider the case where $d \in \{0,1,2\}$. If $d \in \{0, 1\}$ then $\codeg \ge n - 2$, so $a \ge 2$. The graph $G \setminus \{u\}$ has at most two missing edges, so by induction it has a fractional triangle decomposition $\omega'$, which is a triangle packing in $G$ with uncovered weight at most $1 \le a$. If $d = 2$ then $\codeg \ge n - 3$, so $a \ge 1$. If $a \ge 2$, we can again apply the induction hypothesis to conclude that there is a fractional triangle decomposition in $G'$, which is a fractional triangle packing in $G$ with uncovered weight at most $2 \le a$. Finally if $d = 2$ and $a = 1$, then $G[N(u)]$ consists of two adjacent vertices. We think of the single edge in $G[N(u)]$ as a Hamilton cycle with one missing edge, and repeat the above argument.

				From now on, we assume that $\codeg(u) \le (n+a)/3$ for every vertex $u$.
				Let $Z$ be the set of vertices $u$ with $\codeg(u) = 0$, let $U := V(G) \setminus Z$, and write $m := |Z|$.

			\subsection{Case 2. $m \in \{0, 1, 2, 3\}$}

				Let $K$ be the set of vertices $u$ with $\codeg(u) \ge 3$, let $L$ be the set of vertices $u$ with $\codeg(u) = 2$, and denote $k := |K|$ and $\ell := |L|$.

				\begin{claim} \label{claim:m-two-matching-var}
					$2k + \ell \ge m + 3$.
				\end{claim}
				
				\begin{proof} 
					Suppose that $2k + \ell \le m + 2$. Then
					\begin{align*} 
						2(n - 4 + a) = 
						\sum_{u \in U} \codeg(u) 
						& \le k \cdot \frac{n+a}{3} + 2\ell + n - m - k - \ell \\
						& \le k \cdot \frac{n+a}{3} + m - 2k + 2 + n - m - k \\
						& = k \cdot \frac{n+a}{3} + 2 - 3k + n.
					\end{align*}
					It follows that
					\begin{equation*} 
						n \le \frac{30 + (k - 6) \cdot a - 9k}{3 - k}.
					\end{equation*}
					If $k = 0$, we obtain $n \le 10 - 2a \le 10$; if $k = 1$, we have $n \le 15 - 2.5a - 4.5 \le 10.5$; and if $k = 2$, we have $n \le 30 - 4a - 18 \le 12$. Either way, we reach a contradiction to the assumption that $n \ge 14$, thus proving the claim. 
				\end{proof}

				Let $M_1, M_2$ be two edge-disjoint matchings between $Z$ and $K \cup L$ that cover $Z$, such that every vertex in $L$ is covered by at most one of the two matchings. By \Cref{claim:m-two-matching-var}, such a matching exists.
				Indeed, by \Cref{claim:m-two-matching-var} (using $m \le 3$), there exist sets $S_1, S_2 \subseteq K$ and $T_1, T_2 \subseteq L$, such that $T_1$ and $T_2$ are disjoint, $|S_i| + |T_i| = m$ for $i \in [2]$, and $S_1$ and $S_2$ are disjoint if $m = 1$. Now take $M_1$ to be any perfect matching in $G[Z, S_1 \cup T_1]$, and take $M_2$ to be any perfect matching in $G[Z, S_2 \cup T_2] \setminus M_1$. 
				Write $d_1(u)$ and $d_2(u)$ for the degree of $u$ in $M_1$ and $M_2$, respectively. 

				For $u \in U$ let $G_u$ be the graph obtained from $G$ by removing $u$, removing $z$ if $uz \in M_1$ for some $z \in Z$, and removing the edge $uz$ if $uz \in M_2$ for some $z \in Z$ (note that at most two vertices and at most one edge are removed).
				Define $r(u) = \min\{a, \codeg(u) - 1 - d_1(u) - d_2(u)\}$; so $r(u) \ge 0$ for every vertex $u$, by choice of $M_1$ and $M_2$.
				By definition of $M_1$ and $M_2$, the graph $G_u$ has $n - 1 - d_1(u)$ vertices and the following number of missing edges
				\begin{align*}
					n - 4 + a - \codeg(u) + d_2(u) 
					& = n - 5 - d_1(u) + \left(a - \left(\codeg(u) - 1 - d_1(u) - d_2(u)\right)\right) \\
					& \le n - 5 - d_1(u) - (a - r(u)).
				\end{align*}
				(Here we used the assumption that $M_1$ and $M_2$ are edge-disjoint.)
				By induction, there is a fractional triangle packing $\omega_u$ in $G_u$ with uncovered weight at most $a - r(u)$ that has no heavy triangles. Take $\omega = \frac{1}{|U| - 2} \sum_{u \in U} \omega_u$.
				Note that every edge in $G$ appears in exactly $|U| - 2$ of the graphs $G_u$. It follows that the uncovered weight of $\omega$ is at most
				\begin{equation*} 
					\frac{1}{|U| - 2} \sum_{u \in U} (a - r(u)) 
					\le \frac{1}{|U| - 2} \cdot (|U| a - \sum_{u \in U} r(u))
					\le a,
				\end{equation*}
				where for the second inequality we used the following claim. As every triangle appears in at most $|U| - 2$ of the graphs $G_u$, there are no heavy triangles in $\omega$. The proof of \Cref{thm:n-four-a} in this case would be completed once the following claim is proved.

				\begin{claim} 
					$\sum_{u \in U} r(u) \ge 2a$.
				\end{claim}

				\begin{proof} 
					Suppose that $\sum_{u \in U} r(u) \le 2a - 1$.
					As $r(u) \ge 0$ for every $u \in U$, we have $a \ge 1$.

					Suppose first that $\codeg(u) \le a + 1 + d_1(u) + d_2(u)$ for every vertex $u \in U$.
					Then 
					\begin{align*} 
						2a - 1 
						\ge \sum_{u \in U} r(u)
						& = \sum_{u \in U} (\codeg(u) - 1 - d_1(u) - d_2(u)) \\
						& = 2(n - 4 + a) - (n - m) - 2m \\
						& = n - 8 + 2a - m.
					\end{align*}
					It follows that $n \le 7 + m \le 10$, a contradiction to $n \ge 14$.

					Now suppose that $\codeg(v) \ge a + 2 + d_1(v) + d_2(v)$ for some vertex $v$, implying that $\codeg(u) \le a + 1 + d_1(u) + d_2(u)$ for every $u \in U \setminus \{v\}$ (as otherwise $\sum_u r(u) \ge 2a$). So
					\begin{align*} 
						2a - 1
						\ge \sum_{u \in U} r(u) 
						& = a + \sum_{u \in U \setminus \{v\}} (\codeg(u) - 1 - d_1(u) - d_2(u)) \\
						& \ge a + 2(n - 4 + a) - \frac{n+a}{3} - (n - m - 1) - 2m \\
						& = \frac{2n}{3} + 2a + \frac{2a}{3} - 7 - m.
					\end{align*}
					It follows that
					$n \le 9 - a + \frac{3m}{2} \le 13.5$, a contradiction.
				\end{proof}

			\subsection{Case 3. $a < 4$, $m \ge 4$}

				For $z \in Z$, let $\omega_{z}$ be a fractional triangle packing in $G \setminus \{z\}$ with no heavy triangles and with uncovered weight (exactly) $a+1$ (such a packing exists by induction). We assume that $\omega_{z}$ is symmetric on $Z$, i.e.\ swapping the roles of any two vertices in $Z$ does not affect $\omega_{z}$ (this can be achieved by averaging over all packings obtained by permutations of $Z \setminus \{z\}$). Similarly, we assume that $\omega_{z'}$ can be obtained from $\omega_z$ by swapping the roles of $z$ and $z'$, for every $z, z' \in Z$.
				Let $\phi_{z}$ be an edge-weighting, of total weight $1$, corresponding to weight uncovered by $\omega_{z}$ (namely, $\sum_{e \in E(G \setminus \{z\})}\phi_z(e) = 1$, and for every edge $e$ in $G \setminus \{z\}$, $\omega_{z}(e) + \phi_z(e) \le 1$); we again assume that $\phi_{z}$ is symmetric with respect to $Z$. Let $\psi_{z}$ be a weighting on $G \setminus \{z\}$ defined by $\psi_{z}(e) = 1 - \omega_{z}(e) - \phi_{z}(e)$ for every edge $e$ in $G \setminus \{z\}$. Write $\gamma := \phi_{z}(zz')$ for some distinct $z, z' \in Z \setminus \{z\}$; $\alpha_u = \phi_{z}(uz)$ for $u \in U$ and $z \in Z \setminus \{z\}$; and $\beta_{e} = \phi_{z}(e)$ for every edge $e$ in $U$ (note that $\gamma$ and $\alpha_u$ are well-defined, by the symmetry with respect to $Z$). Define $\alpha = (m - 1) \sum_{u \in U} \alpha_u$ and $\beta = \sum_{e \in E(G[U])} \beta_{e}$.
				Then
				\begin{equation} \label{eq:missing-weight}
					\binom{m-1}{2} \gamma + \alpha + \beta = 1.
				\end{equation}

				In order to find the required fractional triangle packing in $G$, we use two approaches. In the first one we consider the graphs $G_u$ for $u \in U$, and modify them slightly by reducing the weight of some edges incident with vertices of $Z$, taking $\codeg(u)$ into account; in particular, the larger $\codeg(u)$ is, the more weight we can remove while still being able to use the induction hypothesis. We then use the available weight on edges incident with $Z$ to compensate for the weight encoded by $\phi_z$, to end up with a packing that has at most $a$ uncovered weight (in contrast to the $a+1$ bound for $\omega_z$). This approach works when $m$ is not too large, because the larger $m$ is, the more extra weight we need to compensate for.

				In the second approach we use the edges in $U \times Z$ to compensate for the extra weight encoded by $\beta_{e}$ for $e \in E(G[U])$, and then cover the remaining weight on these cross edges using triangles with at least two vertices in $Z$. This approach works for larger $m$, because as $m$ grows, the ratio between the weight on edges in $Z$ and the weight on edges in $U \times Z$ increases.

				Define $r(u) = \min\{\codeg(u) - 1, a\}$.

				\begin{claim} \label{claim:m-four}
					$\sum_{u \in U} r(u) \ge 2a$.
				\end{claim}
				\begin{proof} 
					Suppose that $\sum_u r(u) \le 2a - 1$.
					Let $k$ be the number of vertices $u \in U$ with $\codeg(u) \ge a + 1$; then $k \le 1$. 
					We have
					\begin{align*} 
						2a - 1 \ge \sum_{u \in U} r(u) 
						& \ge \sum_{u \in U} (\codeg(u) - 1) - k \left(\frac{n+a}{3} - 1\right) + k a \\
						& = 2(n - 4 + a) - (n - m) - \frac{kn}{3} - \frac{ka}{3} + k + ka \\
						& = \frac{(3-k)n}{3} - (8 - k) + 2a + \frac{2ka}{3} + m.
					\end{align*}
					If $k = 0$, we obtain $n \le 7 - m \le 7$; and if $k = 1$, we obtain $n \le 9 - 3m/2 - a \le 9$. Either way, this is a contradiction to $n \ge 14$.
				\end{proof}

				Let $\sigma : U \to \Z^{\ge 0}$ be such that $\sigma(u) \le r(u)$ and $\sum_u \sigma(u) = 2a$; such a weight assignment exists by \Cref{claim:m-four}. 
				Let $H$ be an auxiliary bipartite graph, with vertex sets $X$ and $Y$, where $X = \{u_0 : u \in U\} \cup \{\zeta\}$, and $Y = \{u_1 : u \in U\}$, and edge set $X \times Y \setminus \{u_0 u_1 : u \in U\}$. We think of $u_0$ and $u_1$ as representing $u$, and of $\zeta$ as representing $Z$.
				We assign a weight $\tau(x)$ to every vertex $x \in V(H)$, as follows.
				\begin{align*} 
					\tau(x) = \left\{
						\begin{array}{ll}
							m \sum_{v \in U} \beta_{uv} & x = u_0 \text{ for some $u \in U$} \\
							\frac{m}{2} \cdot \alpha & x = \zeta \\
							\codeg(u) - 1 - \sigma(u) & x = u_1 \text{ for some $u \in U$}.
						\end{array}
						\right.
				\end{align*}
				(If $uv$ is not an edge, $\beta_{uv} = 0$.)

				A \emph{fractional matching} in $H$ is an assignment $\nu : E(H) \to \R^{\ge 0}$ such that $\sum_{w \in V(H)} \nu(vw) \le \tau(v)$ for every $v \in V(H)$. We say that a fractional matching $\nu$ \emph{saturates} $V$ if $\sum_{w \in V(H)} \nu(vw) = \tau(v)$ for every $v \in V$.

				\begin{claim} \label{claim:matching}
					If $m \le n - 8$, then there is a fractional matching in $H$ that saturates $X$.	
				\end{claim}

				\begin{proof} 
					By a fractional version of Hall's theorem, it suffices to show that every set $A \subseteq X$ satisfies $\tau(N(A)) \ge \tau(A)$.
					As $N(A) = Y$ for every $A \subseteq X$ except for $A = \emptyset$ or $A = \{u_0\}$ for some $u \in U$, it suffices to check that $\tau(Y) \ge \tau(X)$ and $\tau(Y \setminus \{u_1\}) \ge \tau(u_0)$ for every $u \in U$.
					\begin{align*}
						\tau(Y) 
						& = \sum_{u \in U} (\codeg(u) - 1 - \sigma(u)) = 2(n - 4 + a) - (n - m) - 2a = n + m - 8. \\
						\tau(X) 
						& = m \sum_{u \in U} \sum_{v \in U} \beta_{uv} + \frac{m}{2} \cdot \alpha
						= 2m \beta + \frac{m}{2} \cdot \alpha
						\le 2m,
					\end{align*}
					by \eqref{eq:missing-weight}. 
					Thus, as $m \le n - 8$, we have $\tau(Y) \ge \tau(X)$.

					Fix $u \in U$. Then
					\begin{align*} 
						& \tau(u_0)
						= m \sum_{v \in U} \beta_{uv} 
						\le m, \\
						& \tau(Y \setminus \{u_1\})
						\ge \tau(Y) - \codeg(u) + 1
						\ge n + m - 8 - \frac{n+a}{3} + 1 
						\ge \frac{2n}{3} + m - 8 
						\ge m,
					\end{align*}
					using \eqref{eq:missing-weight}, $2n/3 \ge 28/3 > 8$ and $a \le 3$. Thus, $\tau(Y \setminus \{u_1\} \ge \tau(u_0)$ for every $u \in U$, completing the proof of \Cref{claim:matching}.
				\end{proof}
				Consider a fractional matching as in \Cref{claim:matching}, and let $\nu_u(v)$ be the weight of the edge $v_0 u_1$ for $u, v \in U$, and let $\nu_u(\zeta)$ be the weight of the edge $\zeta u_1$.
				Then
				\begin{align*} 
					& \sum_{v \in U} \nu_u(v) + \nu_u(\zeta) \le \tau(u_1) = \codeg(u) - 1 - \sigma(u),\\
					& \sum_{u \in U} \nu_u(v) = \tau(u_0) = m \sum_{u \in U}\beta_{uv},\\
					& \sum_{u \in U} \nu_u(\zeta) = \tau(\zeta) = \frac{m}{2} \cdot \alpha.
				\end{align*}
				Let $G_u$ be the weighted graph obtained from $G \setminus \{u\}$ by decreasing the weight of $vz$ (from $1$) by $\nu_u(v)/m$ for $v \in U \setminus \{v\}$ and $z \in Z$, and decreasing the weight of $zz'$ by $\nu_u(\zeta)/ \binom{m}{2}$ for every distinct $z, z' \in Z$. The missing weight in $G_u$ is 
				\begin{align*} 
					n - 4 + a - \codeg(u) + \sum_{v \in U} \nu_u(v) + \nu_u(\zeta) 
					& \le n - 4 + a - \codeg(u) + \codeg(u) - 1 - \sigma(u) \\
					& = (n - 1) - 4 + a - \sigma(u).
				\end{align*}
				Thus, by the induction hypothesis and \Cref{lem:integer-to-fractional}, there is a fractional triangle packing $\omega_u$ in $G_u$ with no heavy triangles and with uncovered weight at most $a - \sigma(u)$; let $\psi_u(e)$ be the uncovered weight at $e$, for any edge $e$ in $G_u$.
				
				Let $\omega'$ be a fractional triangle packing defined as follows, for distinct $u, v \in U$ and $z, z', z'' \in Z$,
				\begin{equation*} 
					\omega'(uvz) = \beta_{uv} \qquad \qquad
					\omega'(uzz') = \alpha_u \qquad \qquad
					\omega'(zz'z'') = \gamma,
				\end{equation*}
				and
				\begin{align*} 
					\omega = \frac{1}{n-2} \left( \sum_{v \in V(G)} \omega_v + \omega'\right) \qquad \qquad
					\psi = \frac{1}{n-2} \sum_{v \in V(G)} \psi_v.
				\end{align*}
				
				\begin{claim} \label{claim:m-four-end}
					\hfill
					\begin{enumerate} [label = \rm(\alph*)]
						\item \label{itm:total-weight}
							$\omega(e) + \psi(e) = 1$ for every edge $e$ in $G$,
						\item \label{itm:psi-bounded}
							$\sum_{e \in E(G)} \psi(e) \le a$.
					\end{enumerate}
				\end{claim}
				\begin{proof} 
					Recall that $\sum_{e} \psi_z(e) \le a$ for every $z \in Z$, and $\sum_e \psi_u(e) \le a - \sigma(u)$ for $u \in U$. 
					It follows that 
					\begin{equation*}
						\sum_{v \in V(G),\,\, e \in E(G)} \psi_v(e) \le na - \sum_u \sigma(u) = (n - 2)a,
					\end{equation*}
					implying that $\sum_e \psi(e) \le a$, as required for \ref{itm:psi-bounded}.

					Let $e$ be an edge in $G[U]$. We consider three cases: $e = uv$ for $u, v \in U$; $e = uz$ for $u \in U$ and $z \in Z$; and $e = zz'$ for $z, z' \in Z$. 
					In the first case, 
					\begin{align*} 
						(n - 2)(\omega(e) + \psi(e))
						& = \sum_{z \in Z} (\omega_z(e) + \psi_z(e)) + \sum_{w \in U \setminus \{u, v\}} (\omega_w(e) + \psi_w(e)) + \omega'(e) \\
						& = m(1 - \beta_{e}) + (n - m - 2) + m \beta_{e} \\
						& = n - 2.
					\end{align*}
					In the second case,
					\begin{align*} 
						(n - 2)(\omega(e) + \psi(e)) 
						& = \sum_{z' \in Z \setminus \{z\}} (\omega_{z'}(e) + \psi_{z'}(e)) + \sum_{v \in U \setminus \{u\}} (\omega_v(e) + \psi_v(e)) + \omega'(e) \\
						& = (m - 1)(1 - \alpha_u) + (n - m - 1) - \frac{1}{m} \sum_{v \in U \setminus \{u\}} \nu_v(u) + \sum_{v \in U \setminus \{u\}} \beta_{uv} + (m - 1)\alpha_u \\
						& = n - 2.
					\end{align*}
					And in the third case,
					\begin{align*} 
						(n - 2)(\omega(e) + \psi(e))
						& = \sum_{z'' \in Z \setminus \{z, z'\}} (\omega_{z''}(e) + \psi_{z''}(e)) + \sum_{u \in U} (\omega_u(e) + \psi_u(e)) + \omega'(e) \\
						& = (m - 2)(1 - \gamma) + n - m - \frac{1}{\binom{m}{2}} \sum_u \nu_u(\zeta) + (m - 2)\gamma + \sum_{u \in U} \alpha_u \\
						& = n - 2.
					\end{align*}
					We conclude that $\omega(e) + \psi(e) = 1$ for every $e \in E(G)$, as required for \ref{itm:total-weight}.
				\end{proof}
				By \Cref{claim:m-four-end}, $\omega$ is a fractional triangle packing in $G$ with uncovered weight at most $a$. There are no heavy triangles in $\omega$, as every triangle appears in at most $n - 2$ of the packings $\omega'$ and $\omega_v$ for $v \in V(G)$, and none of these packings have a heavy triangle. This completes the proof of \Cref{thm:n-four-a} in this case when $m \le n - 8$.

				We now assume that $m \ge n - 7$.

				Define, for distinct $u, v \in u$ and $z, z', z'' \in Z$,
				\begin{align} \label{eq:defn-omega-prime} 
					\begin{split}
						& \omega'(uvz) = \beta_{uv} \\
						& \omega'(uzz') = \frac{1}{m-1} \left(1 + (m - 1) \alpha_u - \sum_{w \in U} \beta_{uw}\right) \\
						& \omega'(zz'z'') = \frac{1}{m-2} \left(2 + (m - 2) \gamma - \frac{n - m}{m - 1} - \frac{\alpha}{m-1} + \frac{2\beta}{m - 1} \right).
					\end{split}
				\end{align}
				Note that $\omega'(T) \ge 0$ for every triangle $T$ in $G$. Indeed, this clearly holds for $T = uvz$ for some $u, v \in U$ and $z \in Z$, as $\beta_{uv} \ge 0$. Next, if $T = uzz'$ for $u \in U$ and $z, z' \in Z$, then, as $\sum_v \beta_{uv} \le 1$ (by \eqref{eq:missing-weight}), we indeed have $\omega'(T) \ge 0$.
				Finally, if $T = zz'z''$ for $z, z', z'' \in Z$, it suffices to show that
				\begin{equation*} 
					2 \ge \frac{n-m}{m-1} + \frac{\alpha}{m-1}.
				\end{equation*}
				As $\alpha \le 1$ (by \eqref{eq:missing-weight}), it suffices to show that
				\begin{equation*} 
					0 \le 2(m - 1) - (n - m) - 1 = 3m - 3 - n.
				\end{equation*}
				Recall that $m \ge n - 7$, so we have $3m - 3 - n \ge 2n - 24 > 0$, as required.

				Define 
				\begin{align} \label{eq:defn-omega-psi} 
					\begin{split}
						& \omega = \frac{1}{m} \left(\sum_{z \in Z} \omega_z + \omega'\right) \\
						& \psi = \frac{1}{m} \sum_{z \in Z} \psi_z.
					\end{split}
				\end{align}
				\begin{claim} \label{claim:m-four-end-two}
					\hfill
					\begin{enumerate} [label = \rm(\alph*)]
						\item \label{itm:psi-bounded-two}
							$\sum_{e \in E(G)} \psi(e) \le a$,
						\item \label{itm:total-weight-two}
							$\omega(e) + \psi(e) = 1$ for every $e \in E(G)$.
					\end{enumerate}
				\end{claim}
				\begin{proof} 
					Recall that $\sum_{e \in E(G \setminus \{z\})} \psi_z(e) \le a$ for every $z \in Z$, \ref{itm:psi-bounded-two} follows from the definition of $\psi$.
					
					For \ref{itm:total-weight-two}, we consider three cases: $e = uv$ with $u, v \in U$; $e = uz$ with $u \in U$ and $z \in Z$; and $e = zz'$ with $z, z' \in Z$. In each of these cases we will show that $m \cdot \omega(e) = m$.
					In the first case we have
					\begin{equation*} 
						m \cdot \omega(e) = m(1 - \beta_{e}) + m \beta_{e} = m.
					\end{equation*}
					In the second case,
					\begin{equation*} 
						m \cdot \omega(e) 
						= (m - 1)(1 - \alpha_u) + \sum_{v \in U} \beta_{uv} + \left(1 + (m - 1)\alpha_u - \sum_{v \in U} \beta_{uv} \right)
						= m.
					\end{equation*}
					And in the third case,
					\begin{align*} 
						m \cdot \omega(e)
						= (m - 2)(1 - \gamma) \,\,+\,\, & \frac{n - m}{m - 1} + \alpha - \frac{2\beta}{m - 1} \\
						+\,\, 2 + (m - 2)\gamma \,\,-\,\, & \frac{n - m}{m - 1} - \alpha + \frac{2\beta}{m - 1} = m,
					\end{align*}
					completing the proof of \Cref{claim:m-four-end-two}.
				\end{proof}

			\subsection{Case 4. $a = 4$, $m \ge 4$}				

				Fix a non-edge $xy$ (so $x, y \in U$). For $z \in Z$, define $G_z$ to be the graph obtained from $G$ by removing the vertex $z$ and adding the edge $xy$. So $G_z$  has $n-1$ vertices and $n - 5 + a$ missing edges, thus by induction there is a fractional triangle packing $\omega_z'$ on $G_z$ that has uncovered weight at most $a$, and has no heavy triangles. We assume that $\omega_z'$ is symmetric on $Z \setminus \{z\}$, and that $\omega_z'$ can be obtained from $\omega_{z'}'$ by swapping the roles of $z$ and $z'$ for every $z, z' \in S$.
				Let $\psi_z$ be the edge-weighting corresponding to the weight uncovered by $\omega_z'$. Let $\phi_z$ be the edge-weighting defined by $\phi_z(vx) = \phi_z(vy) = \omega_z'(vxy)$ for $v \in V(G) \setminus \{x, y\}$. Let $\omega_z$ be the fractional triangle packing obtained from $\omega_z'$ by changing the weight of triangles containing $xy$ to $0$.
				Define $\gamma = \phi_z(z'z'')$, $\alpha_u = \phi_z(uz')$, and $\beta_{uv} = \phi_z(uv)$, for distinct $z, z', z'' \in Z$ and distinct $u, v \in U$, and write $\alpha = (m - 1)\sum_{u \in U} \alpha_u$ and $\beta = \sum_{e \in E(G[U])} \beta_{e}$.
				Then
				\begin{enumerate}[label = \rm(\roman*)] 
					\item
						$\omega_z(e) + \psi_z(e) + \phi_z(e) = 1$ for every $e \in E(G \setminus \{z\})$.
					\item
						$\sum_{e \in E(G[U])} \psi_z(e) \le a$.
					\item \label{itm:phi-two}
						$\sum_{e \in E(G[U])} \phi_z(e) = \binom{m-1}{2}\gamma + \beta + \alpha \le 2$.
					\item \label{itm:beta-one}
						$\sum_{v \in U} \beta_{uv} \le 1$ for every $u \in U$.
				\end{enumerate}
				To see \ref{itm:phi-two}, note that $\sum_{v} \omega_z'(vxy) \le 1$, thus $\sum_e \phi_z(e) = 2 \sum_v \omega_z'(vxy) \le 2$. Let $u \in U$. If $u = x$ or $u = y$, then $\sum_v \beta_{uv} = \sum_v \omega'_z(vxy) \le 1$; and if $u \neq x, y$, then $\sum_v \beta_{uv} = 2\omega_z'(uxy) \le 1$, by the assumption that $\omega_z'$ does not have heavy triangles; \ref{itm:beta-one} follows. We note that \ref{itm:beta-one} is the reason why we introduced the assumption that there are no heavy triangles. 
				
				We follow the two approaches introduced in the previous case. One main difference is the definition of $\phi_z$ (which is necessary because we cannot use the induction hypothesis for $a+1$, as we did in the previous case), which manifests itself in the upper bound of $2$ in \ref{itm:phi-two}, replacing the upper bound of $1$ that we had previously. This implies that in the first approach we need to compensate for more `extra' weight, thus restricting the range of $m$'s for which the approach works. In order to cover all possible values of $m$, we capitalise on the larger value of $a$, which allows us to remove more weight from the graphs $G \setminus \{u\}$ with $u \in U$. The exact details make this case somewhat technical.  
				For convenience, we reverse the order of the two approaches.

				Let $\omega'$ be a fractional triangle packing defined as in \eqref{eq:defn-omega-prime} from the previous case.
				We note that $\omega'(T) \ge 0$ for every triangle $T$ with at least one vertex in $U$. Indeed, as $\beta_{uv} \ge 0$ for every $u, v \in U$, this holds for $T$ with two vertices in $U$; and if $T$ has one vertex in $U$, the non-negativity follows from \ref{itm:beta-one}. If $T$ has three vertices in $Z$, then $\omega'(T) \ge 0$ if $2(m - 1) - (n - m) - \alpha + 2\beta \ge 0$. As $\alpha + \beta \le 2$, it suffices to have
				\begin{equation} \label{eq:ineq-first-approach} 
					3m \ge n + 4 - 3 \beta.
				\end{equation}
				If \eqref{eq:ineq-first-approach} holds, we define $\omega$ and $\psi$ as in \eqref{eq:defn-omega-psi}. The proof of \Cref{thm:n-four-a} can then be completed following the proof of \Cref{claim:m-four-end-two}. Thus, from now on, we assume that \eqref{eq:ineq-first-approach} does not hold.

				As before, put $r(u) = \min\{a, \codeg(u) - 1\}$.
				\begin{claim} \label{claim:sum-r}
					$\sum_{u \in U} r(u) \ge 2a$.
					Moreover, if $3m \ge n - 7$ then $\sum_{u \in U} r(u) \ge 2a + 6$.
				\end{claim}

				\begin{proof} 
					The proof of \Cref{claim:m-four} can be repeated here to show that $\sum_{u \in U} r(u) \ge 2a$.

					For the second part, suppose that $\sum_{u \in U} r(u) \le 2a + 5$. 
					Let $k$ be the number of vertices $u$ with $\codeg(u) \ge a + 1$. If $k \ge 4$ we have $\sum_{u \in U} r(u) \ge 4a \ge 2a + 5$ (as $a = 4$), so we assume that $k \le 3$.
					\begin{align*} 
						2a + 5 \ge \sum_{u \in U} r(u) 
						& \ge ka + \sum_{u \in U} (\codeg(u) - 1) - k \left( \frac{n+a}{3} - 1 \right) \\
						& = \frac{2ka}{3} + 2(n - 4 + a) - (n - m) - \frac{kn}{3} + k \\
						& \ge \frac{2ka}{3} + \frac{(3 - k)n}{3} - 8 + k + 2a + \frac{n - 7}{3} \\
						& = \frac{2ka + (4 - k)n - 31 + 3k}{3} + 2a,
					\end{align*}
					using $3m \ge n - 7$.
					It follows that
					\begin{equation*} 
						(4 - k)n \le 46 - 2ka - 3k.
					\end{equation*}
					If $k = 0$ we obtain $4n \le 46$;
					if $k = 1$, we have $3n \le 46 - 2a - 3 = 35$;
					if $k = 2$, we have $2n \le 46 - 4a - 6 = 24$;
					and if $k = 3$, we obtain $n \le 46 - 6a - 9 = 13$.
					Either way, as $n \ge 14$, we reached a contradiction.
				\end{proof}
				 
				Define
				\begin{equation*} 
					\rho = \left\{
						\begin{array}{ll} 
							0 & 3m \le n - 8 \\
							\min\{6, m\beta\} & 3m \ge n - 7.
						\end{array}
					\right.
				\end{equation*}
				Let $\sigma$ be a function $\sigma: U \to \Z^{\ge 0}$ such that $\sigma(u) \le r(u)$ for every $u \in U$ and $\sum_{u \in U} \sigma(u) = 2a + \ceil{\rho}$; note that such $\sigma$ exists by \Cref{claim:sum-r}.

				Let $H$ be a bipartite auxiliary graph with vertex sets $X := \{u_0 : u \in U\} \cup \{\zeta\}$ and $Y := \{u_1 : u \in U\}$ and edges $(X \times Y) \setminus \{u_0 u_1: u \in U\}$.
				Define
				\begin{equation*} 
					\tau(x) = \left\{
						\begin{array}{ll}
							m \left(1 - \frac{\rho}{\beta m}\right) \cdot \sum_{v \in U} \beta_{uv} & x = u_0 \text{ for some $u \in U$} \\
							\frac{m}{2} \cdot \alpha & x = \zeta \\
							\codeg(u) - 1 - \sigma(u) & x = u_1 \text{ for some $u \in U$}.
						\end{array}
					\right.
				\end{equation*}

				\begin{claim} \label{claim:matching-rho}
					There is a fractional matching in $H$ that saturates $X$.
				\end{claim}

				\begin{proof} 
					As in the proof of \Cref{claim:matching}, in order to prove that the required matching exists, it suffices to show that $\tau(Y) \ge \tau(X)$ and $\tau(Y \setminus \{u_1\}) \ge \tau(u_0)$ for every $u \in U$. 
					\begin{align*} 
						\tau(Y)
						& = \sum_{u \in U} (\codeg(u) - 1 - \sigma(u)) 
						= 2(n - 4 + a) - (n - m) - 2a - \ceil{\rho} 
						= n + m - 8 - \ceil{\rho} \\
						\tau(X)
						& = m \left(1 - \frac{\rho}{\beta m}\right) 2 \beta + \frac{m}{2} \cdot \alpha
						= 2m\beta - 2 \rho + \frac{m}{2} \cdot \alpha
						\le 4m - 2\rho.
					\end{align*}
					If $3m \le n - 8$ and $\rho = 0$, we have
					\begin{equation*} 
						\tau(Y) - \tau(X) 
						\ge n + m - 8 - 4m \ge 0,
					\end{equation*}
					as required.
					If $3m \ge n - 7$ and $\rho = 6$, we have
					\begin{align*} 
						\tau(Y) - \tau(X)
						& = n + m - 14 - 2m\beta + 12 - \frac{m}{2} \cdot \alpha \\
						& \ge n + m - 2 - 2m\beta - \frac{m}{2} \cdot (2 - \beta) \\
						& \ge \frac{\left( 2(n - 2) - 3m\beta \right)}{2}
					\end{align*}
					By the assumption that \eqref{eq:ineq-first-approach} does not hold, we have
					\begin{equation*} 
						3m\beta \le \beta(n + 4 - 3\beta) 
						\le 2(n - 2),
					\end{equation*}
					where the last inequality holds as $\beta(n + 4 - 3\beta)$ is increasing when $\beta \in [0, 2]$ (the derivative is $n + 4 - 6\beta \ge n - 8 > 0$), and is thus maximised at $\beta = 2$.
					It follows that $\tau(Y) \ge \tau(X)$ in this case.
					Finally, if $\rho = \beta m$, we have $\tau(u_0) = 0$ for every $u \in U$. Hence,
					\begin{equation*} 
						\tau(Y) - \tau(X)
						= n + m - 8 - \ceil{\rho} - \frac{\alpha m}{2} 
						\ge n + m - 14 - m 
						\ge 0,
					\end{equation*}
					where we used the inequalities $\rho \le 6$, $n \ge 14$ and $\alpha \le 2$. 
					We have thus verified that $\tau(Y) \ge \tau(X)$ for all possible values of $\rho$.
					
					We now show that $\tau(Y \setminus \{u_1\}) \ge \tau(u_0)$ for every $u \in U$. Note that when $\rho = \beta m$, $\tau(u_0) = 0$ for every $u \in U$, so this folds trivially. Next, suppose that $\rho \in \{0, 6\}$. Then, using \ref{itm:beta-one},
					\begin{equation*} 
						\tau(u_0) = \left(m - \frac{\rho}{\beta}\right) \sum_{w \in U} \beta_{uw} 
						\le m - \frac{\rho}{2}.
					\end{equation*}
					Thus, if $\rho = 0$, 
					\begin{align*} 
						\tau(Y \setminus \{u_1\}) - \tau(u_0)
						& = \tau(Y) - \codeg(u) + 1 + \sigma(u) - m \\
						& \ge n + m - 8 - \frac{n + a}{3} + 1 - m \\
						& = \frac{2n}{3} - \frac{a}{3} - 7\\
						& \ge \frac{2n - 25}{3}
						\ge 0,
					\end{align*}
					as $n \ge 14$. 
					Finally, consider the case $\rho = 6$. Before continuing, we modify $\sigma$, and before that, we note that there are at most eight vertices $u$ with $\codeg(u) \ge (n+a)/3 - 2$. Indeed, otherwise
					\begin{equation*} 
						2n = 2(n - 4 + a) = \sum_{u \in U} \codeg(u) \ge 8\left(\frac{n+a}{3} - 2\right) = \frac{8n - 16}{3},
					\end{equation*}
					a contradiction to $n \ge 14$. We now modify $\sigma$ so that $\sigma(u) \ge 2$ for every $u \in U$ with $\codeg(u) \ge (n+a)/3 - 2$; and $\sum_u \sigma(u) = 2a + 6 = 14$ (by the above argument such $\sigma$ exists). We thus have $\codeg(u) - \sigma(u) \le (n+a)/3 - 2$ for every $u \in U$, so
					\begin{align*} 
						\tau(Y \setminus \{u_1\}) - \tau(u_0)
						& = \tau(Y) - \codeg(u) + \sigma(u) + 1 - m + \frac{\rho}{2} \\
						& \ge n + m - 8 - \rho - \frac{n+a}{3} + 3 - m + \frac{\rho}{2} \\
						& = \frac{2n}{3} - \frac{a}{3} - 5 - \frac{\rho}{2} \\
						& = \frac{2n - 28}{3} \ge 0,
					\end{align*}
					as $n \ge 14$, $a = 4$ and $\rho = 6$.
				\end{proof}

				Consider a fractional matching as in \Cref{claim:matching-rho}, define $\nu_u(v)$ to be the weight of the edge $u_0 v_1$ for distinct $u, v \in U$, and define $\nu_u(\zeta)$ to be the weight of $\zeta u_1$ for $u \in U$. 
				Then 
				\begin{align*} 
					& \sum_{v \in U} \nu_u(v) + \nu_u(\zeta) \le \tau(u_1) = \codeg(u) - 1 - \sigma(u),\\
					& \sum_{u \in U} \nu_u(v) = \tau(v_0) = \left(m - \frac{\rho}{\beta}\right)\sum_{u \in U}\beta_{uv},\\ 
					& \sum_{u \in U} \nu_u(\zeta) = \tau(\zeta) = \frac{m}{2} \cdot \alpha.
				\end{align*}
				Let $G_u$ be the graph obtained from $G$ by removing the vertex $u$; decreasing the weight of $vz$, where $v \in U$ and $z \in Z$, by $\nu_u(v) / m$; and decreasing the weight of $zz'$, where $z, z' \in Z$, by $\nu_u(\zeta) / \binom{m}{2}$. Note that the weights of the edges of $G_u$ are non-negative, by \ref{itm:beta-one} and \ref{itm:phi-two}.  
				The missing weight in $G_u$ is
				\begin{align*} 
					n - 4 + a - \codeg(u) + \sum_{v \in U} \nu_u(v) + \nu_u(\zeta) 
					& \le n - 4 + a - \codeg(u) + \codeg(u) - 1 - \sigma(u) \\
					& = (n - 1) - 4 + (a - \sigma(u)).
				\end{align*}
				Thus, by induction and by \Cref{lem:integer-to-fractional}, there is a fractional triangle packing $\omega_u$ in $G_u$ that has no heavy triangles and has uncovered weight at most $a - \sigma(u)$; let $\psi_u$ be the weighting corresponding the to weight uncovered by $\omega_u$.

				Let $\omega'$ be the fractional triangle packing defined as follows, for distinct $u, v \in U$ and $z, z', z'' \in Z$.
				\begin{equation*} 
					\omega'(uvz) = \left(1 - \frac{\rho}{\beta m}\right) \beta_{uv} \qquad \qquad
					\omega'(uzz') = \alpha_u \qquad \qquad
					\omega'(zz'z'') = \gamma.
				\end{equation*}
				Let $\psi'$ be the edge-weighting defined by $\psi'(e) = \rho \beta_{e} / \beta$ if $e = uv$ for $u, v \in U$, and setting $\psi(e) = 0$ otherwise. Define 
				\begin{equation*} 
					\omega = \frac{1}{n - 2}\left( \sum_{v \in V(G)} \omega_v + \omega'\right) \qquad \psi = \frac{1}{n-2}\left( \sum_{v \in V(G)} \psi_v + \psi' \right).
				\end{equation*}
				\begin{claim} \label{claim:m-four-end-three}
					\hfill
					\begin{enumerate} [label = \rm(\alph*)]
						\item \label{itm:omega-psi-one}
							$\omega(e) + \psi(e) = 1$ for every edge $e$ in $G$,
						\item \label{itm:psi-bounded-a}
							$\sum_{e \in E(G)} \psi(e) \le a$.
					\end{enumerate}
				\end{claim}
				\begin{proof} 
					Recall that $\sum_e \psi_v(e) \le a - \sigma(v)$ for every $v \in U$ (setting $\sigma(z) = 0$ for $z \in Z$). Thus 
					\begin{align*} 
						\sum_{v \in V(G), \,\, e \in E(G)} \psi_v(e) + \sum_{e \in E(G)} \psi'(e)
						& \le na - \sum_{u \in U} \sigma(u) + \sum_{e \in E(G[U])} \frac{\rho \beta_{e}}{\beta} \\
						& = na - 2a - \ceil{\rho} + \rho 
						\le (n - 2)a,
					\end{align*}
					thus proving \ref{itm:psi-bounded-a}.
					The rest of the proof is very similar to that of \Cref{claim:m-four-end}; we omit further details.
				\end{proof}
				This completes the proof of \Cref{thm:n-four-a}.
			\end{proof}

\bibliography{triangle-packing}
\bibliographystyle{amsplain}

	\appendix

\end{document}